\documentclass[oneside]{amsart}
\usepackage{amsthm,graphicx}
\usepackage{amssymb}
\usepackage{caption}
\usepackage{subcaption}
\counterwithin{equation}{section}
\theoremstyle{definition}
 \newtheorem{theorem}{Theorem}
 
 \newtheorem{proposition}[equation]{Proposition}
 \newtheorem{definition}[equation]{Definition}
 \newtheorem{remark}[equation]{Remark}
 \newtheorem{lemma}[equation]{Lemma}
 \newtheorem{corollary}[equation]{Corollary}
 \newtheorem{example}[equation]{Example}

\newcommand\RR{\mathbb R}

\newcommand\ZZ{\mathbb Z}
\newcommand\NN{\mathbb N}

\newcommand\G{\Gamma}
\newcommand{\dist}{\mathrm{dist}}

\numberwithin{equation}{section}

\usepackage[warn]{textcomp}

\begin{document}

\date{\today}
\title{Sandpile solitons in higher dimensions}
\author[N. Kalinin]{Nikita Kalinin}\thanks{Guangdong Technion-Israel Institute of Technology}\thanks{ORCID: 0000-0002-1613-5175. Research is supported by the Russian Science Foundation grant \textnumero 20-71-00007.}  

\address{Guangdong Technion-Israel Institute of Technology \\ 241 Daxue Road, Jinping District, Shantou, Guangdong Province, China}

\email{nikaanspb\{at\}gmail.com}

\keywords{37B15, 31A05, 14T05, 28A80, 35B36, sandpile model, discrete harmonic functions, solitons, husking, superharmonic functions, cellular automata, tropical geometry, strings}
\begin{abstract}
Let $p\in\ZZ^n$ be a primitive vector and $\Psi:\mathbb Z^n\to \mathbb Z, z\to \min(p\cdot z, 0)$. The theory of {\it husking} allows us to prove that there exists a pointwise minimal function among all integer-valued superharmonic functions equal to $\Psi$ ``at infinity''.

We apply this result to sandpile models on $\ZZ^n$.  We prove existence of so-called {\it solitons} in a sandpile model, discovered in 2-dim setting by S. Caracciolo, G. Paoletti, and A. Sportiello and studied by the author and M. Shkolnikov in previous papers. We prove that, similarly to 2-dim case, sandpile states, defined using our husking procedure, move changeless when we apply the sandpile wave operator (that is why we call them solitons).

We prove an analogous result for each lattice polytope $A$ without lattice points except its vertices. Namely, for each function $$\Psi:\mathbb Z^n\to \mathbb Z, z\to \min_{p\in A\cap \ZZ^n}(p\cdot z+c_p), c_p\in \ZZ$$ there exists a pointwise minimal function among all integer-valued superharmonic functions coinciding with $\Psi$ ``at infinity''. The laplacian of the latter function corresponds to what we observe when solitons, corresponding to the edges of $A$, intersect (see Figure~\ref{fig_soliton}). 

{\bf keywords:} discrete harmonic functions, discrete superharmonic functions, sandpiles, solitons 
\end{abstract}

\maketitle

\section{Introduction. Sandpile patterns on $\ZZ^n$}
 Consider $\ZZ^n$ as a graph with vertices $z\in\ZZ^n$. If the Euclidean distance between $z,z'\in\ZZ^n$ is one, we connect $z,z'$ by an edge and write $z\sim z'$. A sandpile {\it state} is a function $\phi:\ZZ^n\to \ZZ_{\geq 0}$; $\phi(z)$ can be thought of the number of sand grains in $z\in\ZZ^n$. We can {\it topple} a vertex $z$ by sending $2n$ grains from $z$ to its neighbors, i.e. we subtract $2n$ grains from $z$ and each neighbor $z'\sim z$ gets one grain. If $\phi(z)\geq 2n$, such a toppling is called {\it legal}. A {\it relaxation} is doing legal topplings while it is possible, the result of the relaxation of $\phi$ is denoted by $\phi^\circ$, it does not depend of the order of topplings.  The toppling function of a relaxation is the function $\ZZ^n\to \ZZ_{\geq 0}$ which at a vertex $z$ is equal to the number of topplings performed at $z$ during this relaxation. A state $\phi$ is {\it stable} if $\phi\leq 2n-1$ everywhere.

\begin{definition}
\label{def_waves}
Let $z,z'\in \ZZ^n$ be such that $\phi(z)=\phi(z') = 2n-1, z\sim z'$. By {\it sending a wave} from $z$ we mean making a toppling at $z$, followed by the relaxation. We denote the obtained state by $W^z\phi$.
\end{definition}

Note that after the first toppling the vertex $z$ has $-1$ grain, and $z'$ has $2n$ grains, so $z'$ subsequently topples and $z$ has a non-negative number of grains again. If we start with a stable state, then during a wave each vertex topples at most once, because it has not enough grains to topple the second time if all its neighbors toppled only once.

\begin{definition}
Let $q\in\ZZ^n\setminus\{0\}$. A state $\phi$ is called $q$-{\it movable}, if there exists $z_0\in\ZZ^n$ such that $W^{z_0}\phi(z) = \phi(z+q)$ for all $z\in\ZZ^n$ and it is not true that $\phi(z) = \phi(z+q)$ for all $z$. For $L$ being a rank $n-1$ sublattice in $\ZZ^n$, a state $\phi$ is called $L$-{\it periodic} if $\phi(z)=\phi(z+p)$ for each $z\in\ZZ^n, p\in L$. A state $\phi$ is called  {\it hyperplane-shaped} of direction $q\in\ZZ^n\setminus \{0\}$ if there exist constants $c_1,c_2$ such that the set $\{z\vert \phi(z)\ne 2n-1\}$ belongs to the set $\{z\vert  c_1\leq q\cdot z\leq c_2\}$. 
\end{definition}

The first aim of this paper is to classify all {\bf periodic hyperplane-shaped} movable states on $\ZZ^n$, we call them {\it solitons}. Recall that a vector $q\in\ZZ^n$ is called {\it primitive} if there exists no $k\in \ZZ, k>1, q'\in \ZZ^n$ such that $q=k\cdot q'$.

\begin{theorem}
\label{th_main}
For each primitive vector $q\in\ZZ^n$ there exists a unique (up to a translation in $\ZZ^n$) movable state which is hyperplane-shaped of direction $q$. Furthermore, this state is $q$-movable.
\end{theorem}

Such periodic patterns for $n=2$ in sandpiles were studied under the name ``$(p,q)$-webs'' by S. Caracciolo, G. Paoletti, and A. Sportiello in their work~\cite{firstsand}, see also Section 4.3 of~\cite{CPS} and Figure 3.1 in~\cite{book}, Figure 9a in~\cite{sadhu2011effect}. Experiments reveal that these patterns appear in many sandpile pictures and are self-reproducing under the action of waves. That is why we call these patterns {\it solitons}. 
\begin{center} 
\begin{figure}[tph]
\includegraphics[width=0.4\textwidth]{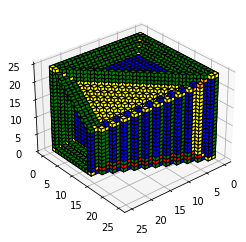}
\includegraphics[width=0.4\textwidth]{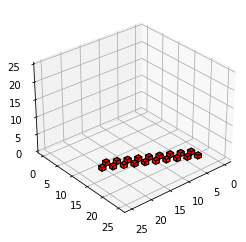}
\includegraphics[width=0.4\textwidth]{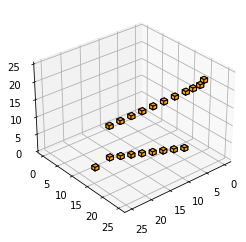}
\includegraphics[width=0.4\textwidth]{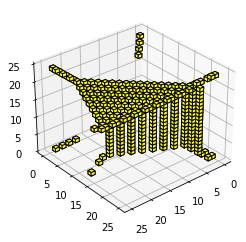}
\includegraphics[width=0.4\textwidth]{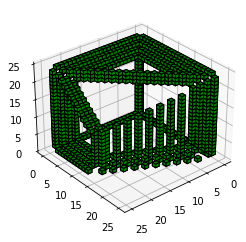}
\includegraphics[width=0.4\textwidth]{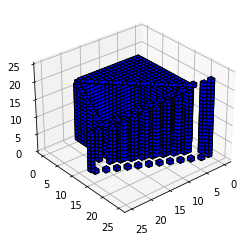}
\caption{We consider a sandpile in $\ZZ^3$ on a graph defined by $x+y+z\leq 50, x+2y\leq 50, 0\leq x,y,z,\leq 25$ where we start with $5$ grains everywhere and added one grain to the point $(4,5,6)$. On the first picture the final result of a relaxation is presented. Red cells represent vertices with zero grains, orange -- with one grain, yellow -- two grains, green -- three grains, blue -- four grains, and cells with five grains are transparent. Other pictures show each color separately. Note that the visible part of the picture (cells with less than five grains) is composed out of planar parts (we will call them solitons) and triples of solitons  intersecting by ``edges'' (the most visible edge is that containing cells with two grains, i.e. these colored in orange).}
\label{fig_soliton}
\end{figure}
\end{center} 

The fact that the solitons appear as ``cutting the corners'' of piece-wise linear functions was predicted by T. Sadhu and D. Dhar in~\cite{sadhu2012pattern}. We introduce a suitable definition of a {\it husking} procedure (Definition~\ref{def_thetan}). This paper is a generalisation of \cite{us_solitons} where similar results were obtained for $n=2$. The study of sandpiles in dimension more than two often lacks generalisations of results known in two dimensions. For example, there is no classification of patterns via quadratic forms similar to \cite{pegden2017stability,levine2013apollonian} in dimensions more than two. 

An example of a soliton for $n=3$ can be found in Figure~\ref{fig_soliton}: solitons represent planar ``faces'' in the picture. The soliton of direction $(0,0,1)$, i.e. that parallel to $xy$ coordinate plane (can be visible on the top part of the picture), has width one and composed of cells with four grains. The soliton of direction $(1,2,0)$ (the right part of the picture) is composed of cells with four, three, and two grains. The soliton of direction $(1,1,1)$ (the center of the picture) is composed of cells with four and two grains.


The second aim of this paper is to describe interactions of solitons in higher dimensions. Consider a lattice polytope $A\subset \RR^n$ without lattice points except its vertices. We prove that for each function $$\Psi:\mathbb Z^n\to \mathbb Z, z\to \min_{p\in A\cap \ZZ^n}(p\cdot z+c_p), c_p\in \ZZ$$ there exists a pointwise minimal function $\psi$ among all integer-valued superharmonic functions equal to $\Psi$ ``at infinity''.
Solitons correspond to the case when $A$ is an interval with lattice endpoints and without lattice points inside.
 The laplacian of $\psi$ corresponds to what we observe when solitons, corresponding to the edges of $A$, intersect. For example, in Figure~\ref{fig_soliton} we see intersections of three planar parts (solitons), i.e. the ``edge'' where we have a repeating pattern which contains cells with two grains (orange) among other cells. The set of orange cells is easy to distinguish, and it lives along a one dimensional ``edge''. 

The plan of this paper is as follows. In order to construct solitons and to study their interactions we define a ``husking'' operation diminishing an integer-valued discrete superharmonic function $\Psi$ while preserving its values ``at infinity'', i.e. we construct a decreasing sequence of functions $$\Psi\to (\Psi)_1\to(\Psi)_2\to\dots,$$  where $\Psi=(\Psi)_k$ ``at inifinity'') and prove a number of properties of this procedure, the most remarkable of which is that husking preserves monotonicity. Then, using periodicity of the function $$\Psi: z\to \min (p\cdot z,0), p\in\ZZ^n$$ along $L=\{v\in\ZZ^n\vert v\cdot p=0\}$ we descend $\Psi$ on $\ZZ^n/L$. Then  we show that if $(\Psi)_k\ne(\Psi)_{k+1}$ for a big $k$, then $(\Psi)_k$ is harmonic on a large part of $\ZZ^n/L$, and it should be linear on that part due an upper bound by a linear function. But then such a linearity would contradict the monotonicity property. Thus there exists $N$ such that $(\Psi)_k=(\Psi)_N$ for all $k>N$. In this case we say that the sequence of huskings of $\Psi$ stabilises.

Stabilisation of huskings is then used to prove Theorem~\ref{th_main}, via the Least Action principle for the toppling functions for waves. Namely, $2n-1+\Delta(\Psi)_N$ is the unique soliton in direction $p$.

Then we consider the case when $A\subset \RR^n$ is a two-dimensional polygon and later study three dimensional $A$, while proving several rather technical lemmata. Finally the proof for a three dimensional $A$ does not differ very much from the general case which we consider in the last chapter.

Let $\Psi(z)=\min_{p\in A}(p\cdot z+ c_p), c_p\in\ZZ$. If the intersection of the convex hull of $A$ in $\RR^n$ with $\ZZ^n$ consist only of the vertices of this polytope then the sequence of huskings of $\Psi$ also stabilises.

This results is proven for $n=2$ in \cite{us_solitons}. The main difficulty in generalising our proofs from \cite{us_solitons} to higher dimensions can be seen in three dimensions. While for solitons the proof is essentially the same, the case when the linear span of $A$ is two dimensional is analogous to the case of triads in \cite{us_solitons}, the case when the dimension of the linear span of $A$ has dimension three is substantially different.  When we consider a function $\Psi:\ZZ^3\to\ZZ$ as above, it is not true that $\Psi-(\Psi)_1$ has a finite support,  quite the contrary, $\Psi-(\Psi)_1$ is not zero near all points where $\Delta(\Psi)$ is not zero (and this set is the corner locus of $\Psi$ seeing as a function $\RR^3\to\RR$, the simplest example is $\Psi(x,y,z) = \min (x,y,z,0)$). So, instead of $\Psi$ we consider a partial husking $\tilde\Psi$ of it, constructed using the husking of functions corresponding to the faces of the convex hull of $A$. Them we need to prove that the support of $$\tilde\Psi-(\tilde\Psi)_k$$ grows at most linearly in $k$. The proof of this fact requires more ideas than we used in the two-dimensional setting.

The main motivation for this work is my desire to generalise the results of \cite{us} to higher dimensions. Namely, if we consider a large lattice polytope in $\ZZ^n$, put $2n-1$ grain to every lattice point, add grains to the points $p_1,\dots,p_l$, and relax this state, then the set of points with less than $2n-1$ grains in the final state $\phi$ is very close to certain tropical hypersurface passing through $p_1,\dots p_l$. The solitons represent hyperplane-like pictures of $\phi$, while the stabilised huskings for generic $A$ represent $\phi$ near vertices of the corresponding tropical hypersurface. The corresponding work on the tropical side is written in \cite{kalinin2021shrinking}.

The authors thank Mikhail Shkolnikov for discussions and an anonymous referee for questions and suggestions.

\section{Husking of integer valued superharmonic functions}
\label{sec_husk}
The discrete Laplacian $\Delta$ of a function $F:\ZZ^n\to\RR$ is defined as $$\Delta F(z) = -2nF(z)+\sum_{z'\sim z}F(z').$$ A function $F$ is called {\it harmonic} (resp., {\it superharmonic}) on $A\subset \ZZ^n$ if $\Delta F(z)=0$ (resp., $\Delta F(z)\leq 0$) for each $z\in A$.

\begin{remark}
Note that making a toppling at $z_0\in\ZZ^n$ in a state $\phi$ produces a state $\phi' = \phi+\Delta 1_{z_0}$ where $1_{z_0}(z)$ is equal to one if $z_0=z$ and is equal to zero otherwise. In general, if $H:\ZZ^n\to \ZZ_{\geq 0}$ is the toppling function of a relaxation of $\phi$ then $\phi^\circ = \phi+\Delta H$.
\end{remark}

\begin{lemma}[Least Action principle for waves, \cite{FLP, us_solitons}]
Let $\phi$ be a stable state on $\ZZ^n$ and $H^{z_0}$ be the toppling function of the relaxation caused by sending a wave from $z_0\in\ZZ^n$. Then $H^{z_0}$ is the pointwise minimal function among the functions $H$ such that $H\geq 0$, $\phi+\Delta H\leq 2n-1$ and $H(z_0)=1$.
\end{lemma}

\begin{lemma}
\label{lem_minharmonic}
If $F,G$ are two superharmonic functions on $A\subset\ZZ^n$, then $\min(F,G)$ is a superharmonic function on $A$.
\end{lemma}

\begin{proof}
Let $z\in A$. Without loss of generality, $F(z)\leq G(z)$. Then, $\Delta \min(F,G)(z)\leq \Delta F(z)\leq 0.$
\end{proof}

\begin{definition} 
For a function $F$,  the {\it deviation set} $D(F)$  is the set of points where $F$ is not harmonic, i.e. 
$$D(F) =\{z\vert  \Delta F(z)\ne 0\}.$$
\end{definition}

For $A\subset \ZZ^n,C>0$, we denote by $B_C(A)\subset \ZZ^n$ the set of points whose Euclidean distance to $A$ is at most $C$.

\begin{definition}
\label{def_thetan}
For $k\in\mathbb{N}$ and a superharmonic function $F:\ZZ^n\to\ZZ$ we define 
$$\Theta_k(F)=\{G:\ZZ^n\to \ZZ\vert  \Delta G \leq 0, F-k\leq G\leq F, \exists C>0, \{F\ne G\}\subset B_{C}(D(F))\}.$$

In plain words, $\Theta_k(F)$ is the set of all integer-valued superharmonic functions $G\leq F$, coinciding with $F$ outside a finite neighborhood of $D(F)$, whose difference with $F$ is at most $k$. 
 Define $(F)_k:\ZZ^n\to\ZZ$ to be the following function $$(F)_k(z)=\min\{G(z)\vert {G\in\Theta_n(F)}\}.$$ We call $(F)_k$ {\it the $k$-husking of $F$}. Note that $(F)_k\geq F-k$. A priori $(F)_k$ does not belong to $\Theta_k(F)$. 
\end{definition}

We now call ``husking'' the process that we used to call ``smoothing'' in \cite{us_solitons} because of confusion: people expect that the result of smoothing is a smooth function, which is not the case in our context, all our functions are from $\ZZ^n$ to $\ZZ$.

\begin{example}
One can easily check that if all the coordinates of $p\in\ZZ^n$ are $0,\pm1$, then $\Psi=(\Psi)_1$ and therefore husking procedure stabilises on the first step.
\end{example}

\begin{lemma}
\label{lemma_path}
Let $F:\ZZ^n\to\ZZ, z_0\in \ZZ^n, F(z_0)\leq k$ and the Euclidean distance between $z_0$ and the set $$\{\Delta F>0\}=\{z\vert  (\Delta F)(z)>0\}$$ be at least $k+2$. Let $z_1\sim z_0$ and $F(z_1)<F(z_0)$. Then there exists a point $z'\in \ZZ^n$ such that $F(z')<0$.
\end{lemma}
\begin{proof}
Indeed, $(\Delta F)(z_1)\leq 0$ and $F(z_0)> F(z_1)$ imply that for some neighbor $z_2$ of $z_1$ we have $F(z_2)<F(z_1)$. Then we repeat this argument for $z_2$ and find its neighbor $z_3$ with $F(z_3)< F(z_2)$, etc. Note that all $z_0,\dots,z_{k+1}$ do not belong to the set $\{\Delta F>0\}$. Finally, we set $z'=z_{k+1}$, since $F(z_{k+1})\leq F(z_0)-(k+1)\leq -1$. 
\end{proof}

\begin{lemma}
\label{lemma_path2}
Let $F:\ZZ^n\to \ZZ$, $z_0\sim z_1\sim \dots\sim z_k$ be a path in $\ZZ^n$ and $F$ be harmonic at all $z_i,0\leq i
\leq k-1$ and $\Delta F(z_k)<0$. Then there exists $i\geq 0$ such that $F(z_0)=F(z_i)$ and $z_i$ has a neighbor $z'$ such that $F(z_i)> F(z')$.
\end{lemma}
\begin{proof}
If $F(z_0)=F(z_k)$ then we may choose $i=k$ and such a neighbor $z'\sim z_k$ exists since $(\Delta F)(z_k)<0$. If not, choose the first $i$ such that $F(z_0)=F(z_i)\ne F(z_{i+1})$ and then use the harmonicity of $F$ at $z_i$. 
\end{proof}

\begin{lemma}
\label{lem_notevident} If two superharmonic functions satisfy $F'\geq F$, then $(F')_k\geq (F)_k$ for each $k\in\ZZ_{\geq 0}$. 
\end{lemma}

\begin{proof}
 Pick any $G'\in \Theta_k(F'),G\in \Theta_k(F)$. In spite of notation we write $\{F>G\}$ instead of $\{z\in\ZZ^n| F(z)>G(z)\}$. We have $G'\geq G$ on $$\{F'-F\geq k\}\cup \{(F')_k=F'\}.$$
Let $\{G'\ne F'\}$ belong to $B_{C_0}(D(F'))$. 

Thus it is enough to prove that $G'\geq (F)_k$ on the set $$A_1=\{F'-F< k\}\cap B_{C_0}(D(F')).$$

Consider the set $$A_2=\{F'-F< k\} \cap \{z\vert \exists z'\sim z, (F'-F)(z)> (F'-F)(z') \}$$

Note that $F'-F$ is a superharmonic function outside of $D(F)$. Since $F'-F\geq 0$ it follows from Lemma~\ref{lemma_path} applied to $F'-F$ that $A_2$ belongs to the $(k+1)$-neighborhood of $D(F)$. 

Next we prove that $A_1\subset B_{C_0}(A_2\cup D(F))$. Indeed, for each point $z_0$ in $A_1$ there exists a path of length at most $C_0$ to the set $D(F')$. If this path intersects $D(F)$, we are done. If not, then Lemma~\ref{lemma_path2} asserts that for a $z_i$ on this path for a certain $z'\sim z_i$, we have $$k>(F'-F)(z_0)=(F'-F)(z_i)>(F'-F)(z')$$ and thus $z_i\in A_2$ and we proved that $A_1\subset B_{C_0}(A_2\cup D(F))$.

Summarising, we obtained that for each $G'\in \Theta_k(F'),G\in \Theta_k(F)$ we have $G'\geq G$ outside $B_{C_0}(A_2\cup D(F))\subset B_{C_0+k+1}(D(F))$. Thus, $\min (G,G')$ belongs to $\Theta_k(F)$, because it coincides with $G$ outside a finite neighborhood of $D(F)$, it is superharmonic, and since $F'-G'\geq k, F-G\geq k, F'\geq F$ we have that $F-\min (G,G')\geq k$. Thus $G'\geq (F)_k$ and so $(F')_k\geq (F)_k$.
\end{proof}

\begin{lemma}
\label{lemma_stabi}
Suppose $\Psi:\ZZ^n\to \ZZ$ and $\tilde\Psi$ belongs to $\Theta_k(\Psi)$ for some $k$. Then the sequence $(\Psi)_k$ stabilises if and only if the sequence $(\tilde\Psi)_k$ stabilises as $k\to \infty$.
\end{lemma}
\begin{proof}
Indeed, $(\Psi)_k\geq (\tilde\Psi)_k$, thus, if $(\tilde\Psi)_k$ stabilises then so does the sequence $(\Psi)_k$. On the other hand, $\Theta_n(\tilde\Psi)\subset \Theta_{n+k}(\Psi)$, thus if $(\Psi)_k$ stabilises, so does the sequence $(\tilde\Psi)_k$.
\end{proof}

\section{Stabilisation of huskings and its corollaries}

The following remark follows from the definition of husking.
\begin{remark}
\label{rem_linear}
Let $F:\ZZ^n\to\ZZ, q,c\in\ZZ^n$. Let $G(z)=F(z)-q\cdot z-c$. Then $(F)_n(z)-q\cdot z-c = (G)_n(z)$.
\end{remark}

\begin{theorem}\label{th_stabilfn}
Pick a primitive vector $p\in\ZZ^n$. Let
\begin{equation}\label{eq_P}
\Psi(z)=\min(0,p\cdot z), z\in\ZZ^n
\end{equation} The sequence of $k$-huskings $(\Psi)_k$ stabilises eventually as $k\to\infty$, i.e. there exists $N>0$ such that $(\Psi)_k\equiv (\Psi)_N$ for all $k>N$. Moreover, $(\Psi)_N$ coincides with $\Psi$ outside a finite neighborhood of $D(\Psi)$.
\end{theorem}

\begin{definition}
The pointwise minimal function in $\bigcup\Theta_k(\Psi)$, which exists by
Theorem~\ref{th_stabilfn}, is called {\it the canonical husking of $\Psi$} and is denoted by $\psi$.
\end{definition}
\begin{remark}
\label{rem_theta}Note that $\Delta \psi\geq -2n+1$ because otherwise we could decrease $\psi$ at a point violating this condition, preserving superharmonicity of $\psi$, and this would contradict to the pointwise minimality of $\psi$ in $\bigcup\Theta_k(\Psi)$.
\end{remark}

Let $q\in \ZZ^n$ be such that $p\cdot q=1$. Note that $\Psi(z+q) = \min (0, p\cdot z+1)$ and  $$\Psi(z-q) = \min (0, p\cdot z-1) = -1+ \min (1, p\cdot z).$$
 
Consider the sandpile state $\phi=2n-1 + \Delta \psi$. By Remark~\ref{rem_theta}, $\phi\geq 0$ and $\phi$ is a stable state because $\psi$ is superharmonic.  Let $z_0\in \ZZ^n$ be a point far from $D(\psi)$. The following corollary says that sending a wave from $z_0$ translates $\phi$ by the vector $\pm q$ depending on the side from where we send the wave. 
\begin{proposition}
\label{cor_wavegp}
In the above conditions, $$(W_{z_0}\phi)(z)=2n-1+\Delta \psi(z\pm q) = \phi(z\pm q)$$ where $W_{z_0}$ is the sending wave from $z_0$ (Definition~\ref{def_waves}) and we choose ``$+$'' if $\psi(z_0) < 0$ and ``$-$'' if $\psi(z_0) = 0$.
 
\end{proposition}
\begin{proof} Let $H^{z_0}$ be the toppling function of the wave from $z_0$. Denote $h(z) =  \psi(z + q)-\psi$ if $\psi (z_0)<0$ and  $h(z) = 1+ \psi(z - q)-\psi$ if $\phi(z_0)=0$.
Since $$W_{z_0}\phi = \phi +\Delta H^{z_0} =2n-1+\Delta(\Psi+H^{z_0}),$$ it is enough to prove that  $H^{z_0}(z) = h(z)$.

 It follows from Lemma~\ref{lem_notevident} that $h(z)$ is non-negative.
 
 $2n-1+\Delta \phi(z\pm q)$ is a stable state. On the other hand, the function $\psi+ h$ coincides with $\phi(z\pm q)$ outside of a finite neighborhood of $D(\phi(z\pm q))$ and is superharmonic. Therefore, by the definition of $\phi(z\pm q)$,  we see that $h\geq  H^{z_0}$ and this finishes the proof.
\end{proof}

\section{Holeless functions}

We need the fact that the set $\{(F)_1\ne F\}$ belongs to a finite neighborhood of $D(F)$ (in particular, this fact implies a pleasurable property $((F)_k)_m = (F)_{m+k}$). Unfortunately, this fact is not true for all superharnomic functions $F$, so we need to restrict the set of functions $F$ that we consider. Namely, we ask for the following technical property prohibiting to have arbitrary large holes in the deviation set.
\begin{definition}
\label{def_holeless}
We say that a function $F:\ZZ^n\to \ZZ$ is {\it holeless} if there exists $C>0$ such that $B_C(D(F))$ contains all the connected components of $\ZZ^n\setminus D(F)$ which belong to some finite neighborhood of $D(F)$. When we want to specify the constant $C$ we write that $F$ is $C$-holeless.
\end{definition}

\begin{example}
\label{ex_small}
$\Psi$ is a holeless function just because $\ZZ^n\setminus D(\Psi)$ has no components which belong to a finite neighborhood of $D(F)$.
\end{example}

\begin{lemma}
\label{lem_finiteneigh}
If $F$ is $C$-holeless, then for each $G\in \Theta_k(F)$ the set $\{F\ne G\}$ is contained in $B_{\max(k,C)}(D(F))$.
\end{lemma}
\begin{proof}
Let $A_k=\{v\in\ZZ^n\vert  G(v)=F(v)-k\}$. If $v\in A_k\setminus D(F)$ then from the superharmonicity of $G$ and harmonicity of $F$ at $v$ we deduce that  all neighbors of $v$ belong to $A_k$. Therefore the connected component of $v\in A_k$ in $\ZZ^n\setminus D(F)$ belongs to $A_k$, which, in turn, belongs to a finite neighborhood of $D(F)$ because there belongs the set $\{F\ne G\}$. Thus $A_k$ belongs to $C$-neighborhood of $D(F)$.  
By the same arguments, for $A_{k-1} = \{G=F-k+1\}$, each point in $A_{k-1}\setminus D(F)$ is contained in the 1-neighborhood of $D(F)\cap A_k$ or, together with its connected component of $\ZZ^n\setminus D(F)$ belongs to $A_{k-1}$, i.e. is contained in $B_C(D(F))$. Then, $A_{k-2}\setminus D(F)$ is contained in the $2$-neighborhood of $D(F)\cap A_k$ or in $1$-neighborhood of $D(F)\cap A_{k-1}$, or in $B_C(D(F))$, etc. 
\end{proof}

\begin{corollary}
\label{cor_finiteness}
If $F$ is $C$-holeless for some $C>0$, then for each $k\geq 0$ the function $(F)_k$ belongs to $\Theta_k(F)$.
\end{corollary}

\begin{corollary}
\label{cor_grows}
For each $k\geq 1$ we have $$\dist\Big(D(\Psi),\big\{\Psi\ne (\Psi)_k\big\}\Big)\leq k,$$
where the distance is the minimum among the Euclidean distances between pairs of points in two sets.
\end{corollary}

\section{husking by steps} 
\label{sec_propertieshusk}

Let $F,G$ be two superharmonic integer-valued functions on $\ZZ^n$. Suppose that $H=F-G$ is non-negative and bounded. Let $m$ be the maximal value of $H$. Define the functions $H_k, k=0,1,\dots, m$ as follows:
\begin{equation}
\label{eq_chi}
H_k(v) =\chi(H\geq k) = \left\{
\begin{aligned}
1,\  & \mathrm{if}\ H(v)\geq k,\\
0,\ & \mathrm{otherwise}.
\end{aligned}
\right.
\end{equation}

\begin{lemma}
\label{lemma_slice}
In the above settings, the function $F-H_m$ is superharmonic.
\end{lemma}
\begin{proof}
Indeed, $F-H_m$ is superharmonic outside of the set  $\{H=m\}$. Look at any point $v$ such that $H(v)=m$. Then we conclude by $$2n(F-H_m)(v)= 2nG(v)+2n(m-1)\geq \sum_{w\sim v} G (w)+2n(m-1)\geq \sum_{w\sim v} (F-H_m)(w).$$
\end{proof}

We repeat this procedure for $F-H_m$; namely, consider $F-H_m-H_{m-1}, F-H_m-H_{m-1}-H_{m-2}$, etc. We have $$H=H_m+H_{m-1}+ H_{m-2}+\dots+ H_1,$$ and it follows from subsequent applications of Lemma~\ref{lemma_slice} that all the functions $F-\sum_{n=m}^{m-k+1} H_n$ are superharmonic, for $k=1,2,\dots,m$. Also, it is clear that $$0 \leq \left(F-\sum_{n=m}^{m-k+1} H_n\right)-\left(F-\sum_{n=m}^{m-k} H_n\right) = H_{m-k}\leq 1$$ at all $v\in\G,k=0,\dots,m$.

Consider a superharmonic function $F.$ We are going to prove that two consecutive huskings (see Definition~\ref{def_thetan}) of $F$ differ at most by one at every point of $\ZZ^n$.

\begin{proposition}
\label{prop_slicingFn}For all $k\in\NN$
$$0\leq (F)_k-(F)_{k+1}\leq 1.$$
\end{proposition}
\begin{proof}
By definition, $(F)_k\geq (F)_{k+1}$ at every point of $\mathbb{Z}^n$. If the inequality $(F)_k-(F)_{k+1}\leq 1$ doesn't hold, then the maximum $M$ of the function $H=(F)_{k}-(F)_{k+1}$ is at least $2$. We will prove that 
$$(F)_k-\chi (H\geq M)\geq F-k.$$

Namely, by Lemma \ref{lemma_slice} the function $(F)_k-\chi (H\geq M)$ is superharmonic. Suppose that $$(F)_k-\chi (H\geq M)< F-k \text{\ at\ a\ point\ $v\in\ZZ^n$}.$$ 
Since the set $\{H\geq 1\}$ contains the set $\{H\geq M\}$, we arrive to a contradiction by saying that, at $v$, $$F-(k+1)>(F)_k-\chi (H\geq M) - \chi (H\geq 1)\geq (F)_{k+1}\geq F-(k+1).$$ 

Therefore $(F)_k-\chi (H\geq M)\in\Theta_k(F)$ which contradicts the minimality of $(F)_k$.

\end{proof}

\begin{corollary}\label{cor_fnmin}
Proposition~\ref{prop_slicingFn} and Lemma~\ref{lem_finiteneigh} imply that for $C$-holeless $F$ the function $(F)_{k+1}$ can be characterized as the point-wise minimum of all superharmonic functions $G$ such that $(F)_k-1\leq G\leq (F)_k$  and $(F)_k-G$ vanishes outside some finite neighborhood of $D((F)_k)$ (recall that the distance between $D(F), D((F)_k)$ is at most $\max(C,k)$). 
In other words, $k$-husking $(F)_k$ of $F$ is the same as $1$-husking of $(k-1)$-husking $(F)_{k-1}$ of $F$.
\end{corollary}

\begin{corollary}
\label{cor_1husking}
In the above assumptions, if $(F)_k\ne (F)_{k+1}$ then there exists $z_0$ such that $(F)_{k+1}(z_0)=F(z_0)-(k+1)$.
\end{corollary}
Indeed, if there is no such a point, then $(F)_{k+1}\geq F-k$ and therefore $(F)_{k+1}=(F)_k$.

\section{Monotonicity while husking}
\label{sec_mono}
\begin{definition}
\label{def_monotone}
Let $e\in\ZZ^n\setminus\{0\}$. We say that a function $F:\ZZ^n\to\ZZ$ is {\it $e$-increasing} if 
\begin{enumerate}
\item $F$ is a husking of a holeless function,
\item $F(z)\leq F(z+e)$ holds for each $z\in\mathbb{Z}^n$,
\item there exists a constant $C>0$  such that for each $z$ with $F(z)=F(z-e)$, the first vertex $z-ke$ in the sequence $z,z-e,z-2e,\dots$, satisfying $F(z-ke)<F(z-(k-1)e)$, belongs to $B_C(D(F))$.
\end{enumerate}
\end{definition}

\begin{example}
\label{ex_monotone}
Note that $\Psi$ is $e$-increasing if and only if $p\cdot e>0$.
\end{example}

\begin{lemma}
If $F$ is $e$-increasing, then $(F)_1,$ the $1$-husking of $F$, is also $e$-increasing. 
\end{lemma}

\begin{proof}

Corollary~\ref{cor_fnmin} gives the property {\bf a)} of Definition~\ref{def_monotone}, because if $F=(G)_n$, and $G$ is holeless, then $(F)_1 = (G)_{n+1}$.
To prove that $(F)_1$ satisfies {\bf b)} in Definition~\ref{def_monotone} we
 argue {\it a contrario}. Let $H= F-(F)_1$. Suppose that the set $$A=\{z\in\mathbb{Z}^n \vert  F(z-e)=F(z),H(z-e)=0, H(z)=1\}$$ is not empty. Since $H\vert _A=1$, we have $A\subset B_C(D(G))$. Consider the set $$B=\{z\vert  H(z)=0, \exists k\in\mathbb{Z}_{> 0}, z+k\cdot e\in A, F(z)=F(z+k\cdot e)\}.$$ 
 
Consider $z\in B$. Since $z+k\cdot e \in A\subset B_C(D(F))$ and $F(z)=F(z+k\cdot e)$, then {\bf c)} in Definition~\ref{def_monotone} impose an absolute bound on $k$  and therefore $B$ belongs to a finite neighborhood of $D(F)$.  Consider the following function $$\tilde F=(F)_1-\sum_{z\in B}\delta_z.$$ 
It is easy to verify that $\tilde F(z)\leq \tilde F(z+e)$ for each $z$. Note that $$\Delta \tilde F(z)\leq \Delta (F)_1(z)\leq 0$$ automatically for all $z\in\ZZ^n\setminus B$. 
 Pick any $z\in B$. Since $z+k\cdot e\in A$ for some $k\in\ZZ_{>0}$, we have
$$2n\cdot\tilde F(z)=2n\cdot(F)_1(z+k\cdot e)\geq \sum_{z'\sim z}(F)_1(z'+k\cdot e)\geq \sum_{z'\sim z}\tilde F(z').$$ 

 Therefore $\tilde F$ is superharmonic, and satisfies $F\geq \tilde F\geq F-1$ by construction, which contradicts to the minimality of $(F)_1$ in $\Theta_1(F)$. 
 
  Finally, by Corollary~\ref{cor_finiteness} the sets $\{F\ne (F)_1\}$ and $D((F)_1)$ belong to $B_C(D(F))$ for some $C>0$. Therefore the fact that $\vert (F)_1-F\vert \leq 1$ (Proposition~\ref{prop_slicingFn}) gives {\bf c)} with the constant $C+\vert e\vert +1$. 
\end{proof}

\begin{corollary}\label{cor_monotf}
Let $e\in\ZZ^n\setminus \{0\}$. If $\Psi$ is $e$-increasing, then $(\Psi)_k$ is also $e$-increasing.
\end{corollary}

\section{Discrete superharmonic integer-valued functions.}
\label{sec_discrete}

\begin{lemma}(\cite{duffin}, Theorem 5)
\label{lemma_duffin} 
There is an absolute constant $C$ with the following property. Let $R>1,z\in\ZZ^n$, and $F:B_R(z)\cap \ZZ^n\to\RR$ be a discrete {\bf non-negative} harmonic function. Let $z'\sim z$, then $$\vert F(z')-F(z)\vert \leq \frac{C\cdot \max_{w\in B_R(z)} F(w)}{R}.$$
\end{lemma}

Morally, this lemma provides an estimate on a derivative of a discrete harmonic function, for $F(z')-F(z)$ can be thought of a discrete derivative of $F$ in the direction $z'-z$.

\begin{lemma}[Integer-valued discrete harmonic functions of sublinear growth]
\label{lemma_harmonic} 
Let $z\in\ZZ^n$ and $\mu>0$ be a constant. Let $R>4\mu C$. For a discrete {\bf integer-valued} harmonic function $F:B_{3R}(z)\cap\ZZ^n\to\ZZ$, the condition $\vert F(z')\vert \leq \mu R$ for all $z'\in B_{3R}(z)$ implies that $F$ is linear in $B_{R}(z)\cap \ZZ^n$. 
\end{lemma}
\begin{proof} Consider $F$ which satisfies the hypothesis of the lemma.
Note that $0\leq F(z')+\mu R\leq 2\mu R$ for $z'\in B_{3R}(z)$ and applying Lemma~\ref{lemma_duffin} for $B_{2R}(z)$ yields $$\vert \partial_\bullet F(z')\vert \leq \frac{C\cdot 2\mu R}{R}=2\mu C \text{,\ for all $z'\in B_{2R}(z)$}.$$ 
By $\partial_\bullet$ we denote any of the discrete partial derivatives, $\partial_i F(z) = F(z+e_i)-F(z)$ where $e_i$ is the $i$-th coordinate vector. 
Then, applying it again for $0\leq \partial_\bullet F(z')+2\mu C\leq 4\mu C$ yields
\begin{equation*}
\vert \partial_{\bullet}\partial_{\bullet}F(z')\vert \leq \frac{4\mu C}{R}<1 \text{,\ for\  $z'\in B_{R}(z)$\ if\ $R>4\mu C$.}
\end{equation*}  Since $F$ is integer-valued, all the derivatives $\partial_{\bullet}\partial_{\bullet}F$ are also integer-valued. Therefore all the second derivatives of $F$ are identically zero in $B_{R}(z)$, which implies that $F$ is linear in $B_{R}(v)$.
\end{proof}

 Let $A$ be a finite subset of $\ZZ^n$, $\partial A$ be the set of points in $A$ which have neighbors in $\ZZ^n\setminus A$. Let $F$ be any function $A\to \ZZ$.
 
\begin{lemma}
\label{lemma_nabla} In the above hypothesis the following equality holds: 
$$\sum_{z\in A\setminus\partial A} \Delta F(z) = \sum_{\substack{z\in \partial A,\\ z'\in A\setminus\partial A, z\sim z'}} \big(F(z)-F(z')\big).$$
\end{lemma}
\begin{proof}
We develop left side by the definition of $\Delta F$. All the terms $F(z)$, except for the vertices $z$ near $\partial A$, cancel each other. So we conclude by a direct computation.
\end{proof}

\section{Proof of Theorem~\ref{th_stabilfn} via reduction to a cylinder}

Recall that for a primitive vector $p\in\ZZ^n$ we consider
\begin{equation}\label{eq_P}
\Psi(z)=\min(0,p\cdot z), z\in\ZZ^n.
\end{equation} 

Our aim is to prove that the sequence of $k$-huskings $(\Psi)_k$ stabilises eventually as $k\to\infty$. 

Let $$L=\{v\in \ZZ^n\vert  v\cdot p=0\}.$$
Note that $\Psi(z) = \Psi(z+v)$ for each $v\in L$. Translations by $L$ preserve graph structure on $\ZZ^n$ therefore the factor-graph $\ZZ^n/L$ (i.e. we say that $z$ is equivalent to $z'$ if and only if $z-z'\in L$) is well-defined. 

Note that if the images of $z,z'\in \ZZ^n$ are different in $\ZZ^n/L$ then $p\cdot z\ne p\cdot z'$. 

Note that $\Psi$ is $p$-monotone, see Definition~\ref{def_monotone}, so all $(\Psi)_k$ are $p$-monotone.

The function $\Psi$ descends to  $\ZZ^n/L$ and its deviation locus  on $\ZZ^n/L$ is a finite set, therefore $\sum_{z\in \ZZ^n/L} \Delta \Psi(z)=P$ is a finite number. 
\begin{lemma}
\label{lemma_periodichuskongs}
For all $k\in\ZZ_{>0}$ huskings $(\Psi)_k$ are $L$-periodic, i.e. for all $v\in L, z\in\ZZ^n$ $$(\Psi)_k(z)=(\Psi)_k(z+v).$$
\end{lemma}
\begin{proof}
Suppose, to the contrary, that $(\Psi)_k(z')>(\Psi)_k(z'+v)$ for some $z'\in\ZZ^n, v'\in L$. It follows from Lemma~\ref{lem_minharmonic} that $F(z)=\min((\Psi)_k(z),(\Psi)_k(z+v'))$ belongs to $\Theta_k(\Psi)$, but 
$F(z')<(\Psi)_k(z')$ which contradicts to the minimality of $(\Psi)_k$ in $\Theta_k(\Psi)$.
\end{proof}
Therefore all $k$-huskings are also $L$-periodic, so it makes sense to consider $k$-huskings of $(\Psi)_k$ on $\ZZ^n/L$.

\begin{proof}[Proof of Theorem~\ref{th_stabilfn}]
Suppose that the sequence of $k$-huskings $(\Psi)_k$ of $\Psi$ do not stabilize as $k\to \infty$. Then the distance (in $\ZZ^n/L$) between the deviation locus $D((\Psi)_k)$ of $(\Psi)_n$ and $D(\Psi)$ grows as $k\to\infty$, i.e. the lengths of the interval  $$I_k= [\min (p\cdot z)\vert  z\in D((\Psi)_k)), \max (p\cdot z\vert  z\in D((\Psi)_k))] $$
can be arbitrary large. Denote $J_0=\{z\in\ZZ^n/L\vert  p\cdot z\in I_k\}$. 

It follows from Lemma~\ref{lemma_nabla} that $\sum_{z\in \ZZ^n/L} \Delta \Psi(z) = \sum_{z\in \ZZ^n/L} \Delta (\Psi)_k(z) = P$. Hence there exists a 
subinterval $I'\subset I_k$, of length at least $\vert I_k\vert /P$ such that $(\Psi)_k$ is harmonic on the set $J_1=\{z\vert  p\cdot z\in I'\}$.
It follows from Lemma~\ref{lemma_harmonic} that $(\Psi)_k$ is linear on $J_1$, because $(\Psi)_k$ on $J_1$ can not grow faster than $p\cdot z$ (because of $-p$ monotonicty of the function $(\Psi)_k-p\cdot z$). Note that $(\Psi)_k$ is $L$-periodic, therefore $(\Psi)_k$ is $cp\cdot z+d$ for some constants $c,d\in\RR$. These constants $c,d$ must be integer, because $(\Psi)_k$ is integer-valued.

We know that $(\Psi)_k$ is $p$-monotone, and $(\Psi)_k  = p\cdot z$ for $\{z\vert  p\cdot z<<0\}$ and $(\Psi)_k=0\cdot z$ for $\{z\vert p\cdot z>>0\}$, therefore $0\leq c\leq 1$. Indeed, if $c<0$ this would contract the $p$-monotonicity, and $c>1$ is prohibited by the same reasoning applied to the function $(\Psi)_k - p\cdot z$ which is $(-p)$-monotone.

Therefore $c=0$ or $c=1$. If $c=0$, then $(\Psi)_k\equiv d$ on $J_1$. Next, $d$ cannot be $0$ since in between of $J_1$ and the region $\{z\vert  p\cdot z>>0\}$ there exists a point with $\Delta (\Psi)_k<0$. If $d<0$ then, again, this contradicts to the $p$-monotonicity of $(\Psi)_k$ on the region from $J_1$ to $\{z\vert p\cdot z>>0\}$. The case $c=1$ is analogous. Thus we arrived to a contradiction, therefore there exists $N$ such that  for all $k\in \ZZ_{\geq 0}$ we have
 $$(\Psi)_N=(\Psi)_{N+k}.$$

Note that $(\Psi)_N$ coincides with $\Psi$ outside a finite neighborhood of $D(\Psi)$, because it is so in $\ZZ^n/L$.
 \end{proof}

\section{Proof of Theorem~\ref{th_main}}
Proposition~\ref{cor_wavegp} implies that for each primitive non-zero $p$ there exists a soliton. So we need to only prove that all solitons can be obtained in this way. Our plan is as follows: we send big number $k$ of waves from a point  far from a soliton and then deduce that the toppling function of this process is essentially the husking of $\min(k,p\cdot z+c)$ for suitable constants $k,c$.

Consider a movable hyperplane-shaped $L$-periodic state $\phi$. Suppose that we are sending waves from a point $z_0\in\ZZ^n$. Choose a primitive vector $p$ such that $p\cdot L=0$. Without loss of generality we may suppose that $D(\phi)\subset \{z\vert  -m<p\cdot z<m\}$ and $p\cdot z_0>>m$ for some $m\in\ZZ_{\geq 0}$.  As in the previous section, let us descend the sandpile states and wave topplings functions to the cylinder $\ZZ^n/L$. 

\begin{lemma}
In the above setting, sending a wave from a point $z_0$ causes no topplings in the set $\{z\vert  p\cdot z<<0\}$.
\end{lemma}
\begin{proof}
Suppose that there is a toppling in the set  $\{z\vert  p\cdot z<<0\}$. Then the whole region  $\{z\vert  p\cdot z<-m\}$ topples. For $k$ big enough send $k$ such waves from $z_0$. Denote the corresponding toppling function by $H^{kz_0}$. Then $H^{kz_0}$ is equal to $k$ in the region $\{z\vert  p\cdot z>>0\}$, $H^{kz_0}=k$ in $\{z\vert  p\cdot z<-m\}$, $\Delta H^{kz_0} \geq 0$ in $ \{z\vert  -m<p\cdot z<m\}$ and $\Delta H^{kz_0}\leq 0$ everywhere else. Let the deviation set of $(W^{z_0})^k\phi$ belong to the set $\{z\in \ZZ^n/L\vert  p\cdot z\in [Ck-m,Ck+m]\}$. Note that for each $z\in \ZZ^n/L$ we have $H^{kz_0}(z)\geq k- C$ because, in the above assumptions, during a wave number $l<k$ a vertex does not topple only if it belongs to $D((W^{z_0})^{l-1}(\phi))$, and the latter moves with some constant speed $C$ (because it is periodic and movable). Take a point $z_1\in \ZZ^n/L$ with $\Delta H^{kz_0}(z_1)<0$. It follows from Lemma~\ref{lemma_path} that there exists $z$ with $H^{kz_0}(z)<k-C-1$, because the set $\{z\vert  \Delta H^{kz_0}(z)>0\}$ is far from $z_1$. This is a contradiction.
\end{proof}

\begin{proof}[Proof of Theorem~\ref{th_main}] Using the notation of the previous lemma, let us send $k$ waves from $z_0$. Then, the toppling function $H^{kz_0}$ is a non-negiative pointwise minimal function $h$ such that $h(z_0)=k$ and $\phi+\Delta h\leq 2n-1$ everywhere. Consider $h$ as the function on $\ZZ^n/L$. It has an upper bound by $\min (n, p\cdot z+m)$. Therefore for big $k$ the toppling function $H^{kz_0}$ is linear in between of $D(\phi)$ and $D((W^{z_0})^k\phi)$, therefore by the same reasoning as in Theorem~\ref{th_stabilfn}, $H^{kz_0}$ is equal to $p\cdot z+c$ on one part of $\ZZ^n/L$ and to $k$ on another, and is a pointwise minimal superharmonic function with these properties. Hence $\{z\vert  \Delta(H^{kz_0})<2n-1$ coincides with a translation of $\psi$ for $\Psi(z)=\min(0,p\cdot z)$. This concludes the proof.

\end{proof}

\section{The case of planar polygon $A$}
Consider a set $A\subset\ZZ^n$ whose linear span has dimension two. Without loss of generality $0\in A$. Suppose that the intersection of the convex hull of $A$ with $\ZZ^n$ is $A$. Denote $\Psi(z)=\min_{p\in A}(p\cdot z + c_p)$ for some constants $c_p\in\ZZ$. Let $L=\{q\in \ZZ^n| \forall p\in A q\cdot p = 0\}$. Then $L$ is a lattice and $\Psi(z)$ is $L$-periodic. Thus we can consider the graph $\ZZ^n/L$. Since $A$ is a convex polygon we can order its vertices as $p_1,p_2\dots,p_m,p_{m+1}=p_1$.  Note that $m=3$ or $4$, since a convex polygon $A\subset \ZZ^2$ with bigger number of vertices in the lattice contains a lattice point except its vertices.

For each two subsequent vertices $p_i,p_{i+1}$ of $A$ we know that a husking of $$\Psi_i(z) = \min(p_i\cdot z+ c_{p_i},p_{i+1}\cdot z+ c_{p_{i+1}})$$ exists, we denote it by $\psi_{p_i,c_i,p_{i+1},c_{i+1}}$.

Consider

$$\tilde\Psi(z) =\min\limits_{i=1,\dots,m} (\psi_{p_i,c_i,p_{i+1},c_{i+1}}).$$
It follows from Lemma~\ref{lemma_stabi} that it is enough to prove that the sequence $(\tilde\Psi)_k$ of huskings  stabilises.

\begin{lemma}
\label{lemma_fin}
The set $D_1=\{z| (\tilde\Psi)(z)\ne (\tilde\Psi)_1(z)\}$ is finite.
\end{lemma}
\begin{proof}
Since all functions $(\tilde\Psi)_k$ are periodic with respect to $L$, we can descend these functions to the factor-graph $G=\ZZ^n/L$. The graph $G$ is essentially two dimensional, in the sense that we can choose coordinates on it such that only two coordinates $x_1,x_2$ take any integer value and all other coordinates are bounded. 
Suppose $D_1$ is not finite. Note that the deviation set $D(\tilde\Psi)$ is contained in a finite neighborhood of the corner locus of the function $\Psi(z)$ (seeing as a function $\Psi:\RR^n\to\RR$, the corner locus of such a piece-wise linear function is the set of points where it is not smooth). In $G$ this finite neighborhood looks as a finite neighborhood of three or four rays $R_i, i=1,\dots, m$ in the plane $x_1,x_2$. The set $D_1$ belongs to $1$-neighborhood of $D(\tilde\Psi)$, so if it is infinite, then $D_1$ can be found arbitrary far from the origin. Each of $\psi_{p_i,c_i,p_{i+1},c_{i+1}}$ is periodic, so we can find two parts of $\psi_{p_i,c_i,p_{i+1},c_{i+1}}$ where $D_1$ looks identically. Then, similar to Lemma 9.6 in \cite{us_solitons}, by taking $(\tilde\Psi)_1(z)$ in between of these two parts and copying it along $\psi_{p_i,c_i,p_{i+1},c_{i+1}}$ we obtain a function which belongs to $\Theta_1(\psi_{p_i,c_i,p_{i+1},c_{i+1}})$ and not equal to it, which contradicts to the definition of $\psi_{p_i,c_i,p_{i+1},c_{i+1}}$. So $D_1$ can not be infinite.
\end{proof}

\begin{theorem}
\label{th_two}
Suppose that $A$ is above. Then the sequence of huskings $(\tilde\Psi)_k$ stabilises as $k\to \infty$.
\end{theorem}

\begin{proof}
Suppose that the sequence $(\tilde\Psi)_k$ does not stabilises. We descend all objects to $G=\ZZ^n/L$ as in Lemma~\ref{lemma_fin}. It follows from Corollary~\ref{cor_1husking} that there exists a vertex $v_0$ such that $(\tilde\Psi)_k(v_0) = (\tilde\Psi)(v_0) -k$ for each $k\geq 1$. Indeed, such vertex exists for every $k$ and to show that one can choose one vertex $v_0$ suitable for all $k$ it is enough to show that the set $\{z| (\tilde\Psi)(z)\ne (\tilde\Psi)_1(z)\}$ is finite, which is so according to Lemma~\ref{lemma_fin}. Then we consider the sets $D_m = \{z| (\tilde\Psi)(z)\ne (\tilde\Psi)_m(z)\}$. We will show that $D_m$ grows at least linearly and at most linearly by $m$. 

Indeed, Lemma~\ref{lemma_fin} implies that the difference between $D_m, D_{m+1}$ is at most constant, since they propagate along the rays in $\ZZ^n/L$, so $D_m$ grows at most linearly by $m$, so for some $R$ we have $D_m\subset B_{Rm}(0)$ for all $m$. On the other hand, for each direction $w\in\ZZ^n$ we can find a linear function $l(z)$ such that $f(z)+l(z)$ is monotone in the direction $w$. Adding a linear function does not change the process of husking (Remark~\ref{rem_linear}). But since $(\tilde\Psi)_k(v_0) = (\tilde\Psi)(v_0) -k$, then $(\tilde\Psi)_k(v_0+w) \leq \min(\tilde\Psi(v_0+w), \tilde\Psi(v_0) -k)$. This proves that $D_m$ grows at least linearly in $m$, for some $r>0$ we have $B_{rm}(0)\subset D_m$.

Note that $\sum_{z\in B_{Rk}(0)} \Delta (\tilde\Psi)_k(z) = \sum_{z\in B_{Rk}(0)} \Delta (\tilde\Psi)(z)\leq Ck$ for some $k$.

It follows from Lemma~\ref{lemma_nabla} that $\sum_{z\in B_{rk}(0)} |\Delta (\tilde\Psi)_k(z)|\leq Ck $ and so one can find a large ball in $B_{rk}(0)$ where $(\tilde\Psi)_k$ is harmonic and so is linear by Lemma~\ref{lemma_harmonic}. By repeating the arguments from the last part of the proof of Theorem~\ref{th_stabilfn} we see that from the monotonicity it follows that the slope of this linear function must be in the convex hull of $A$, but the latter contains no lattice points except the vertices of $A$. So, we arrived to a contradiction, hence the sequence of huskings $(\tilde\Psi)_k$ eventually stabilises. 
\end{proof}

Note that the soliton is periodic with respect to $L$. Let us fix a fundamental domain of the action of $L$ on it, it looks as (1) on Figure~\ref{fig_2}: a soliton belongs to a finite neighborhood of a hyperplaneplane, and images of the fundamental domain under the action of $L$ tiles this neighborhood. 

\section{The case of three dimensional $A$}

Let $n=3$.
Suppose that the dimension of the linear span of $A\subset \ZZ^3$ is three, so the convex hull of $A$ is a lattice polytope. Suppose that the convex hull of $A$ contains no lattice points except vertices of $A$. 

Consider a function $\Psi:\ZZ^3\to \ZZ,$

$$\Psi(z) = \min_{p\in A}(p\cdot z+ c_p), c_p\in \ZZ.$$

For each face $F$ of $A$ we may consider the function
$$\Psi_F(z) = \min_{p\in F}(p\cdot z+ c_p),$$
and from the previous section we know that the sequence of its huskings stabilises on a function $\psi_F(z)$.

Recall that a tropical hypersurface is the set of non-linearity of a function $\min_{p}(p\cdot z+ c_p)$ where $p$ runs by a finite subset of $\ZZ^n, c_p\in\RR$. Figure~\ref{fig_2} represents a typical neighborhood of a vertex in a tropical hypersurface.

\begin{figure}
\includegraphics{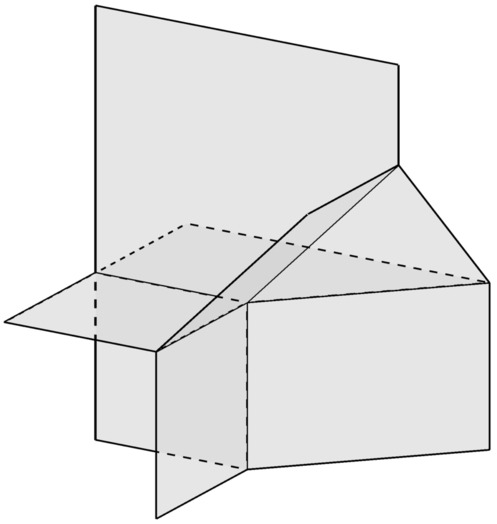}
\includegraphics[scale=0.17]{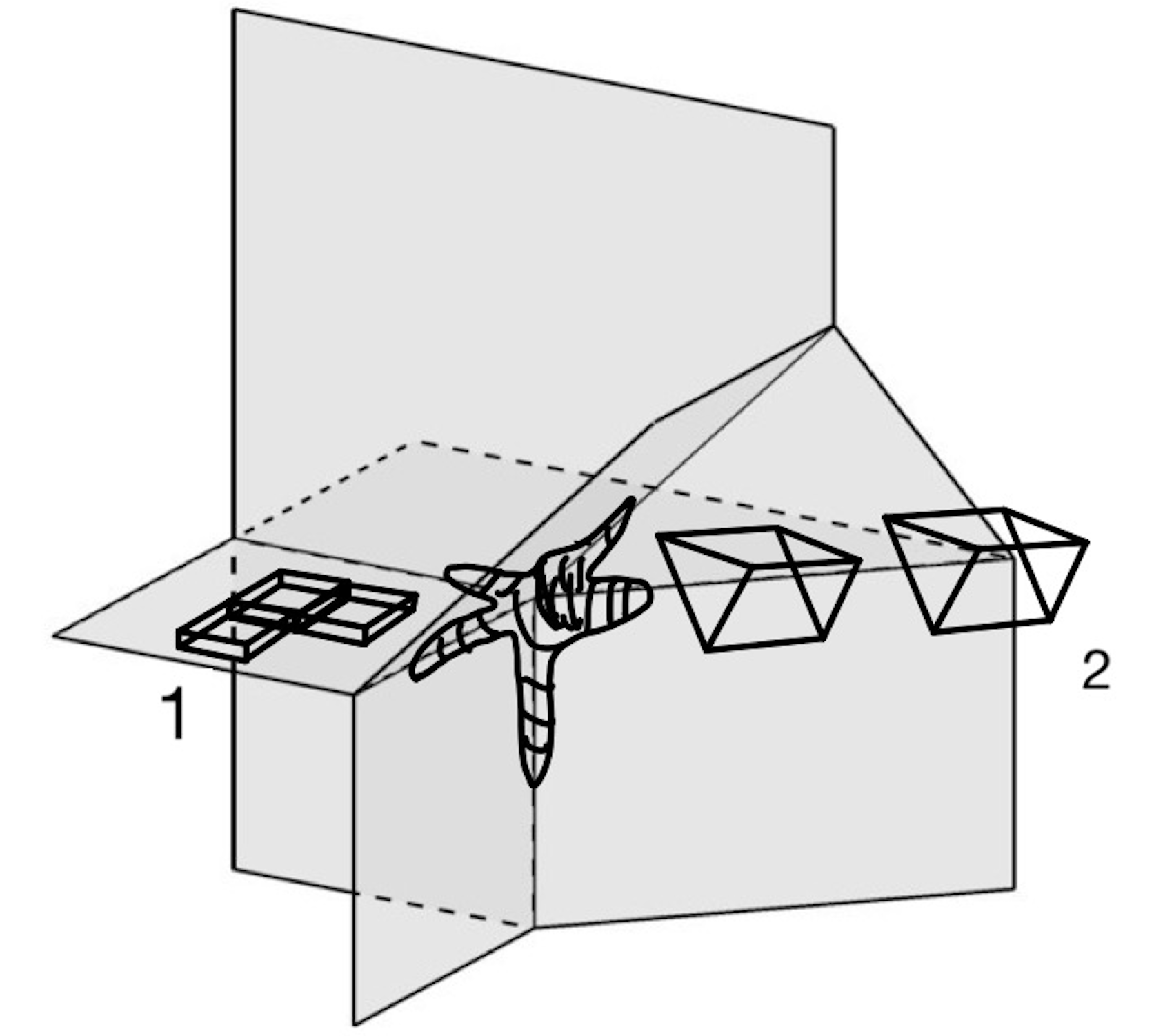}
\caption{Left: An example of a neighborhood of a vertex in a tropical hypersurface, courtesy of \cite{renaudineau2017tropical}. Right: A schematic picture of periodic fundamental domains for solitons (1), fundamental domains along edges (2), and an example of the support of $\tilde\Psi-(\tilde\Psi)_k$, which propagated from the vertex along the edges and faces of a tropical hypersurface. The support of $\tilde\Psi-(\tilde\Psi)_k$ is pictures as a blot. }
\label{fig_2}
\end{figure}

Define the function $\tilde \Psi(z) = \min_F(\psi_F)$ where we $F$ runs by all faces of the convex hull of $A$. As in the previous section, in order to prove stabilisation of huskings of $\Psi$ it is enough to prove the stabilisation of the huskings of $\tilde \Psi$.

\begin{theorem}
The sequence of the huskings $(\tilde \Psi)_k$ stabilised as $k\to\infty$.
\end{theorem}
\begin{proof}
The general schema of the proof is the same, except one new idea. First we need to prove that $D_m=\{z| \tilde \Psi(z)\ne (\tilde \Psi)_m(z)\}$ grows at least linearly in $m$. A proof repeats the beginning of the proof of Theorem~\ref{th_two}: we find a vertex $v_0$ such that $(\tilde \Psi)_m(v_0) = \tilde \Psi(z)-m$ and then use the fact that husking preserves monotonicity. Thus there exists $r>0$ such that $B_{rm}(0)\subset D_m$.
 
A fact that $D_m$ grows at most linearly in $m$ needs finer arguments as follows. As we know, the deviation set of $\tilde \Psi$ consists of solitons (corresponding to the edges of the convex hull of $A$, they live near faces of the corresponding function, see Figure~\ref{fig_2}). Note that each of  these solitons is periodic with respect to a two-dimensional lattice. Thus the fundamental domain with respect to this periodicity has a finite diameter. Let $c$ be the maximum of these diameters for all solitons corresponding to the edges of the convex hull of $A$. 

Let $X\subset \ZZ^3$ be the set of vertices where $\tilde\Psi$ is not equal to none of $\psi_{p,c_p,q,c_q}$ where $pq$ runs by all edges of the convex hull of $A$.  Then there exists a constant $C$ such that if $v\in D_1$ then the distance between $v$ and $X$ is at most $Cc$. Indeed, suppose the contrary. Consider the fundamental domain of the corresponding soliton  where $v$ belongs to. Consider the neighbouring copies the fundamental domains and the restriction of $(\tilde\Psi)_1-\tilde\Psi$ on it. Then we take their neighbouring domains etc. Each time we look at the restriction of $(\tilde\Psi)_1-\tilde\Psi$ on the new copies of the fundamental domains (see (1) on Figure~\ref{fig_2}), we call it $f_1,f_2,f_3,\dots$. If at some step no new functions appeared, i.e. the restriction of $(\tilde\Psi)_1-\tilde\Psi$ on the next belt of fundamental domains coincide with previously constructed $f_1,\dots, f_l$, then we take $\min(f_i)$ and prolong it periodically to a function $f$. Thus we obtain a function $\psi_{p,c_p,q,c_q}+f$ which belongs to $\Theta_1(\psi_{p,c_p,q,c_q})$ and less than $\psi_{p,c_p,q,c_q}$ this contradicts to the minimality of $\psi_{p,c_p,q,c_q}$. Thus, the distance between $X$ and $v$ is a most $c$ times ``the number of $\{-1,0\}$-valued functions of the fundamental domain for this soliton''.

Note that $X$ belongs to the finite neighborhood of the rays (corresponding to the faces of the convex hull of $A$) of the corner locus of $\Psi:\RR^3\to\RR$. So if $D_1$ is infinite, then it prolongs in a finite neighborhood of $X$, so in a finite neighborhood of these rays. Again, if we find two identical pieces of $D_1$ (such as (2) in Figure~\ref{fig_2}) along such a ray (corresponding to a face $F$ of the convex hull of $A$), we would decrease the stabilised husking of $\Theta_F$, which is not possible. So $D_1$ is finite and similarly the distance between $D_m$ and $D_{m+1}$ is a priory bounded and so $D_m$ grows at most linearly by $m$, so there exists $R$ such that $D_m\subset B_{Rm}(0)$.

Then, using the estimate for the laplacian, we see that the number of vertices in $B_{rm}(0)$, where $(\tilde\Psi)_k$ is not harmonic, is linear by $k$, so one can find a large ball in it where $(\tilde\Psi)_k$ is harmonic, and so linear. Finally, the slope (thought of a lattice point) of that discovered linear function would belong to the convex hull of $A$ because of monotonicity properties of husking procedure, but we know that $A$ contains no such points in its convex hull. So we arrived to a contradiction and thus finished the proof.
\end{proof}

\section{General case}

\begin{theorem} Let $A\subset \ZZ^n$ be a finite set. Suppose that $A$ coincides with the set of vertices of the convex hull of $A$ in $\RR^n$. Suppose that the intersection of the convex hull of $A$ with $\ZZ^n$ consists only of $A$. Define

$$\Psi(z) = \min_{p\in A} (p\cdot z+ c_p),$$

where $c_p$ are arbitrary integer numbers. Then the sequence $(\Psi)_k$ of huskings of $\Psi$ eventually stabilises, i.e. there exists $N$ such that for all $k>N$ we have $(\Psi)_k=(\Psi)_N$.
\end{theorem}
\begin{proof} Our proof combines ideas which we used in the partial cases above. Without loss of generality we may suppose that $0\in A$. First, if the dimension of the linear span of $A$ is less than $n$, then instead of $\ZZ^n$ we consider the graph $\ZZ^n/L$ where $L = \{q\in\ZZ^n|q\cdot p=0, \forall p\in A \}$. Then we choose coordinates in $\ZZ^n/L$ and the number of coordinates which may be infinite is equal to the dimension of the linear span of $A$ (all the same as in our proof of Theorem~\ref{th_stabilfn}). So, for simplicity of exposition we assume that the dimension of the linear span of $A$ is $n$.

Then we proceed by induction. We already proved stabilisation of huskings for $n=1,2,3$ and now we describe the general case for arbitrary $n$. Let $F$ be a face of the convex hull of $A$. Let $\Psi_F(z) =\min_{p\in F} (p\cdot z+ c_p)$. We know that the sequence of huskings of $\Psi_F$ stabilises on a function $\psi_F$ by induction hypothesis. Define $\tilde\Psi  = \min_F(\psi_F)$ where $F$ runs by all faces of the convex hull of $A$. 
It is enough to prove stabilisation of $(\tilde\Psi)_k$ as $k\to\infty$. 

First, we show that the support of $(\tilde\Psi)_1-\tilde\Psi$ is finite. Suppose not. As in the proof of three dimensional case, we can prove that the distance between the support of $(\tilde\Psi)_1-\tilde\Psi$ and the codimension one faces of the corner locus of $\Psi$ is a priori bounded. Indeed, we cut the corresponding soliton to copies of the fundamental domain of the periodic action, consider the difference $(\tilde\Psi)_1-\tilde\Psi$ on each of them. Then the distance between any point in the support of $(\tilde\Psi)_1-\tilde\Psi$ and the codimension one faces of the corner locus of $\Psi$ is estimated by the product of the diameter of the fundamental domain and the total number of functions on the fundamental domain with the values $0$ and $-1$. Then, in a finite neighborhood of codimension one faces of the corner locus of $\Psi$ we have similar procedure and establish that the distance from the support of $(\tilde\Psi)_1-\tilde\Psi$ to the codimension two faces is a priori bounded (otherwise we could decrease $\tilde\Psi$ there which contradicts the induction hypothesis). 

So we proved that the support of $(\tilde\Psi)_1-\tilde\Psi$ is finite. So there exists a vertex $v$ such that $\tilde\Psi(v) - (\tilde\Psi)_k(v) = k$, and so the growth of the support of $\tilde\Psi - (\tilde\Psi)_k$ is at least linear in $k$ (the argument is exactly the same as in three-dimensional case).

Then one shows that the support of $\tilde\Psi - (\tilde\Psi)_k$ (shown in Figure~\ref{fig_2} near the vertex) grows at most linear in $k$ in the same way as above we have shown that the support of $\tilde\Psi - (\tilde\Psi)_1$ is finite. Then, as in the three dimensional case, we show that there is a large ball in the support of $\tilde\Psi - (\tilde\Psi)_k$ where $(\tilde\Psi)_k$ is harmonic, and so it is linear. But this contradicts to our hypothesis that the lattice points of the convex hull of $A$ are only the vertices of this convex hull. This finishes the proof.

\end{proof}

\section{Conflict of Interest}
Not applicable.

\bibliographystyle{abbrv}

\end{document}